\documentclass[10 pt]{article}
\usepackage{graphicx}
\usepackage{amsmath,amsthm,latexsym,amssymb,amsfonts,latexsym}
\usepackage[bookmarksnumbered, colorlinks, plainpages]{hyperref}

\newtheorem{theorem}{Theorem}[section]

\newtheorem{lemma}{Lemma}[section]
\newtheorem{corollary}{Corollary}[section]
\newtheorem{example}{Example}[section]
\newtheorem{proposition}{Proposition}[section]
\numberwithin{equation}{section}
\makeatletter
\renewcommand*\env@matrix[1][\arraystretch]{%
  \edef\arraystretch{#1}%
  \hskip -\arraycolsep
  \let\@ifnextchar\new@ifnextchar
  \array{*\c@MaxMatrixCols c}}
\makeatother
\newcommand{\spec}{\operatorname{spec}}
\newcommand{\rank}{\operatorname{rank}}
\newcommand{\trace}{\operatorname{trace}}

\usepackage{caption, subcaption}\usepackage{mathtools}
\usepackage{enumerate}
\usepackage{enumitem}
\usepackage[letterpaper,top=2cm,bottom=2cm,left=3cm,right=3cm,marginparwidth=1.75cm]{geometry}
\usepackage{enumerate}
\usepackage{tikz}
\usepackage{tkz-euclide}
\usepackage{float}
\usepackage{url}
\usetikzlibrary{positioning}
\usepackage{a4wide}

\usepackage{cite}
\usepackage{underscore}
\allowdisplaybreaks
\date{}
\begin{document}
\title{\textbf{On Randi\'c energy of a vertex}}
	\date{}
\author{Idweep J. Gogoi$^1$, J. Buragohain$^2$, A. Bharali$^3$, E. Devi$^4$ \\
	Department of Mathematics, Dibrugarh University, Assam, India$^{1,\;2\;3\;4}$\\
rs_idweepjyotigogoi@dibru.ac.in$^1$\footnote{Corresponding author: rs_idweepjyotigogoi@dibru.ac.in}, jibonjyoti@dibru.ac.in$^2$, a.bharali@dibru.ac.in$^3$,\\ rs_erashikhadevi@dibru.ac.in$^4$}
\maketitle

\begin{abstract}
In 2018, Arizmendi and Juarez introduced the concept of energy of a vertex, a novel approach allowing the total energy of a graph to be expressed as the sum of the energies of its individual vertices. In this article, we extend the notion of energy of a vertex to the context of the Randi\'c matrix. We define the Randi\'c energy of a vertex and explore its mathematical properties through various combinatorial techniques. We derive several upper and lower bounds for the Randi\'c energy of a vertex. Furthermore, we establish that among the connected graphs, the central vertex of a star attains the maximum Randi\'c energy, whereas pendent vertices attain the minimum. Also, we report the Coulson-type integral formula for the Randi\'c energy of a vertex and its applications.
\itshape \bf {AMS Classifications (2010):} 05C50, 05C07, 15A18.
\end{abstract}
\section{Introduction}
Throughout the paper, we consider connected simple graph with at least two vertices. Let $G\coloneqq(V,E)$ be a graph with $n$ vertices and $m$ edges, the vertex set $V(G)\coloneqq\{v_1,\ldots,v_n\}$ and edge set $E(G)\coloneqq\{e_1,\ldots,e_m\}$. The degree of a vertex $v_i$ is denoted by $d_i$. Let $\mathcal{M}_n(\mathbb{C})$ denote the set of all $n\times n$ matrices with complex entries. For $M \in \mathcal{M}_n(\mathbb{C})$, the absolute values of $M$ is denoted by $|M|\coloneqq(MM^*)^{1/2}$, where $M^*$ is the conjugate transpose of the matrix $M$.

Spectral graph theory has significant applications in chemistry, particularly in modeling the energy levels of $\pi$-electrons in conjugated hydrocarbons \cite{Coulson1940}. In 1978, Gutman \cite{Gutman1978} introduced the concept of graph energy, based on the eigenvalues of the adjacency matrix $A(G)$ or simple $A$. If $\delta_1,\ldots,\delta_n$ be the eigenvalues of $A(G)$, then the energy \cite{Gutman1978} of $G$ is
$$\mathcal{E}(G)=\sum_{i=1}^n\left|\delta_i\right|.$$
The concept of graph energy became foundational in chemical graph theory literature. Motivated by its success, various extensions of graph energy have since been proposed using alternative matrices beyond the adjacency matrix.

The matrix $D^{-1/2}A(G)D^{-1/2}$ is called the normalized adjacency matrix of an un-oriented graph $G$, where $D\coloneqq\text{diag}\left(d_1,\ldots,d_n\right)$. This matrix is popularly known as the Randi\'c matrix $R(G)$. If $R(G)\coloneqq (R_{ij})$, then
$$
R_{ij}=\left\{
  \begin{array}{ll}
    \frac{1}{\sqrt{d_id_j}} & \hbox{if $v_i \sim v_j$} \\
    \;\;\;0 & \hbox{otherwise.}
 \end{array}
\right.
$$
Clearly, $R(G)$ is Hermitian. Rodr\'{i}guez \cite{Rodriguez2005} first studied the Randi\'c matrix in 2005. Later, Bozkurt et al. \cite{Bozkurt2010} proposed the Randi\'c energy $\mathcal{RE}(G)$ of a graph $G$ as
$$\mathcal{RE}(G)=\sum\limits_{i=1}^n |\lambda_i|,$$
where $\lambda_1,\ldots,\lambda_n$ are the eigenvalues of $R(G)$. The spectral properties and energy of the Randi\'c matrix have been widely explored, leading to several notable bounds on Randi\'c energy (see \cite{Li2015, Liu2012, Das2014, Das2015, Estrada2017}). In 2018, Arizmendi and Juarez-Romero \cite{Arizmendi2018Romero} introduced a refinement of graph energy known as the energy of a vertex, which was further developed in the same year by Arizmendi et al. \cite{Arizmendi2018}, offering a more localized perspective on graph energy. It is denoted by $\mathcal{E}_G(v_i)$ and is defined as
$$\mathcal{E}_G(v_i)=|A(G)|_{ii},$$
and hence the energy of the graph $G$ is
$$\mathcal{E}(G)=\trace(|A(G)|)=\sum\limits_{i=1}^n |A(G)|_{ii}.$$

The motivation behind the energy of a vertex is to interpret it as the contribution of an individual vertex to the total energy of the graph based on its interaction with other vertices. Following this reasoning, we define the Randi\'c energy of a vertex analogously, offering a localized measure within the framework of the Randi\'c matrix.

\section{Preliminaries}
Consider the positive linear functional $\phi_i:\mathcal{M}_n\rightarrow \mathbb{C}$, which is define for a graph $G$ by $\phi_i\left(R(G)\right)\mapsto R(G)_{ii}$. Using this, we can find that the trace can be decomposed as the sum of positive linear functionals,
$$\trace(R(G))=\phi_1(R(G))+\ldots+\phi_n(R(G)).$$
Let us define
$$\mathcal{RE}_G(v_i)=\phi_i(|R(G)|)=|R(G)|_{ii},$$
where $|R(G)|\coloneqq(R(G)R^*(G))^{1/2}$ is the absolute value of $R(G)$. Observe that $\trace(|R(G)|)=\phi_1(|R(G)|)+\ldots+\phi_n(|R(G)|),$ which implies $\mathcal{RE}(G)=\mathcal{RE}_G(v_1)+\ldots+\mathcal{RE}_G(v_n).$ Hence the Randi\'c energy of the graph $G$ is
$$\mathcal{RE}(G)=\sum\limits_{i=1}^n |R(G)|_{ii}.$$

In \cite{Arizmendi2018}, the authors discussed a method to calculate the energy of a vertex in terms of the eigenvalues and eigenvectors of adjacency matrix. A similar approach can be adopted to compute the Randi\'c energy of a vertex in a graph.

\begin{lemma}\label{Lemma definittion of vetrex energy}
Let $G=(V,E)$ be a connected graph and $\lambda_1\geq\ldots\geq\lambda_n$ are eigenvalues of $R(G)$, then $\forall$ $ v_i\in V$
$$\mathcal{RE}_G(v_i)=\sum\limits_{j=1}^ny^{2}_{ij}|\lambda_j|,$$
where $Y\coloneqq (y_{ij})$ be the orthogonal matrix whose columns are the real orthonormal eigenvectors of $R(G)$.
\end{lemma}
\begin{proof}
The matrix $R(G)$ can be express as $$R(G)=Y\Lambda Y^t,$$
where $Y\coloneqq (y_{ij})$ be an orthogonal matrix whose columns are the real orthonormal eigenvector of $R(G)$ and $\Lambda=\text{diag}\left(\lambda_1,\ldots,\lambda_n\right)$. It can easily be seen that $Y$ satisfies $\sum\limits_{i=1}^ny^2_{ij}=1$ and $\sum\limits_{j=1}^ny^2_{ij}=1$. So, the absolute value of the matrix $R(G)$ is
$$|R(G)|=(RR^*)^{\frac{1}{2}}=Y(\Lambda \Lambda^*)^{\frac{1}{2}}Y^t.$$
Note that $(\Lambda \Lambda^*)^{\frac{1}{2}}=\left(|\lambda_1|,\ldots,|\lambda_n|\right).$
Thus,
\begin{align*}
&|R(G)|_{ii}=\sum\limits_{j,k=1}^nY_{ik}\left[(\Lambda \Lambda^*)^{1/2}\right]_{kj}Y^t_{ji}=\sum\limits_{j=1}^ny_{ij}|\lambda_j|y_{ij}=\sum\limits_{j=1}^ny^{2}_{ij}|\lambda_j|.\qedhere
\end{align*}
\end{proof}

\begin{example}Consider the following graph.
\begin{figure}[H]
  \begin{center}
\begin{tikzpicture}[scale =.8, vertex/.style={circle, draw, minimum size=8mm, inner sep=0pt}]
    \node[vertex, label=below:$v_1$] (1) at (-.4,-2) {\small{0.3382}};
    \node[vertex, label=below:$v_2$] (2) at (1,-1.5) {\small{0.5847}};
    \node[vertex, label=below:$v_3$] (3) at (2.5,-1.5) {\small{0.3382}};
    \node[vertex, label=right:$v_4$] (4) at (1,0) {\small{0.6990}};
    \node[vertex, label=above:$v_5$] (5) at (2,1.5) {\small{0.5468}};
    \node[vertex, label=above:$v_6$] (6) at (0.2,1.7) {\small{0.5468}};
    \node[vertex, label=left:$v_7$] (7) at (-2,0) {\small{0.3172}};

    \draw (1) -- (2) -- (3) -- (4) -- (2) -- (1) -- (4);
    \draw (4) -- (5) -- (6) -- (4);
    \draw (4) -- (7);
\end{tikzpicture}
\caption{Graph with 7 vertices}
\end{center}
\end{figure}
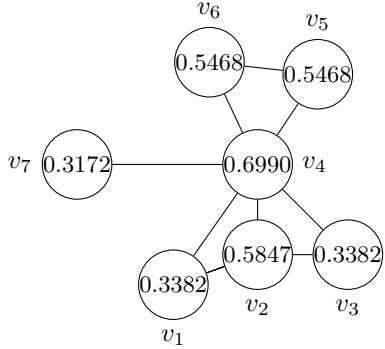

The Randi\'{c} eigenvalues of G are $\{-0.6855, -0.5000, -0.4999, 0, 0.1518, 0.5336, 1\}$ and the corresponding eigenvectors are given by the columns of matrix $Y$ as follows:
\[
Y =
\begin{bmatrix}
-0.4389   & 0.0000 &   -0.2722  & -0.7071  & -0.1375 &   0.3214  &  0.3333\\
    0.3199 &  -0.0000 &   0.6667  &  0.0000 &  -0.2666 &   0.4642 &   0.4082\\
   -0.4389 &   0.0000 &  -0.2722 &   0.7071 &  -0.1375  &  0.3214  &  0.3333\\
    0.5898  &  0.0000 &  -0.4713 &   0.0000 &   0.3046  & -0.0623  &  0.5774\\
   -0.1436  &-0.7071  &  0.1361  &  0.0000  & -0.2526   &-0.5347   & 0.3334\\
   -0.1436  &  0.7071 &   0.1361 &   0.0000 &  -0.2526  & -0.5347  &  0.3334\\
   -0.3512  & -0.0000 &   0.3848 &  -0.0000 &   0.8190  & -0.0476  &  0.2357
\end{bmatrix}
\]
Using Lemma \ref{Lemma definittion of vetrex energy}, we have
\begin{align*}
\mathcal{RE}_{G}(v_1)=& \sum_{j =1}^{7} y_{1j}^2 \; |\lambda_j| \\
  =& (0.4389)^{2}(0.6855) + (0)^{2}(0.5) + (-0.2722)^{2}(0.4999)+ (-0.7071)^{2}(0) \\
  & + (-0.1375)^{2}(0.1518) + (-0.3214)^{2}(0.5336)+ (0.3333)^{2}(1) \\
  \approx &\; 0.3382.
\end{align*}
Similarly, $\mathcal{RE}_{G}(v_2) = 0.5847$, $\mathcal{RE}_{G}(v_3) = 0.3382$, $\mathcal{RE}_{G}(v_4) = 0.6990$, $\mathcal{RE}_{G}(v_5) = 0.5468$, $\mathcal{RE}_{G}(v_6) = 0.5468$, $\mathcal{RE}_{G}(v_7) = 0.3172$.\\
\end{example}

By a similar argument as in the proof of Lemma \ref{Lemma definittion of vetrex energy}, we may obtain the following lemma.
\begin{lemma}\label{Lemma of iith entry of Randic higher power}
Let $\phi_i:M_n\rightarrow \mathbb{C}$ be a positive map define by $\phi_i(R)\mapsto R_{ii}$, then
$$\phi_i(R^k(G))=\sum\limits_{j=1}^ny^{2}_{ij}\lambda_j^k,$$
where $Y\coloneqq(y_{ij})$ is an orthogonal matrix whose columns are given by the eigenvectors of $R(G)$.
\end{lemma}


\section{Some bounds on Randi\'c energy of a vertex}
In this section, we study various upper and lower bounds for the Randi\'c energy of a vertex. We begin with some known results which are associated to our findings.

\begin{theorem}\label{Randic spectral radius}\cite{Liu2012}
The Randi\'c spectral radius is 1.
\end{theorem}
\begin{theorem}\label{Randic spectral radius multiplicity}\cite{Liu2012}
Let $G$ be a graph, if $G$ is connected then 1 is a simple eigenvalue of $R(G)$.
\end{theorem}
\begin{theorem}\label{Randic eigenvalue symmetricity}\cite{Li2015}
Let $G$ be a connected graph, then the spectrum of $R(G)$ is symmetric about 0 if and only if $G$ is a bipartite graph.
\end{theorem}

From Theorems \ref{Randic spectral radius}, \ref{Randic spectral radius multiplicity} and \ref{Randic eigenvalue symmetricity}, we have the following result for the supremum Randi\'c energy of a vertex.

\begin{theorem}
Let $G=(V,E)$ be a connected graph with at least two vertices, then $\forall$ $v_i\in V$,
$$\mathcal{RE}_G(v_i)\leq 1,$$
equality holds if $G$ is a star graph, with vertex $v_i$ being the center of the star.
\end{theorem}
\begin{proof}
Let $\lambda_1\geq\ldots\geq\lambda_n$ are eigenvalues of $R(G)$, then replacing each $|\lambda_i|$ by 1 (Randi\'c spectral radius) in Lemma \ref{Lemma definittion of vetrex energy}, we have
$$\mathcal{RE}_G(v_i)\leq\sum\limits_{j=1}^ny^{2}_{ij}=1.$$

For the equality we must have either $(i)$ $\lambda_i\in \{-1,1\}$, $\forall$ $i\in \{1,\ldots,n\}$ or $(ii)$ the $i^{th}$ entries of the eigenvectors corresponding to eigenvalues with $|\lambda_i|\neq 1$ are all zero. Since $G$ is connected, $\lambda_1=1$ is a simple eigenvalue of $R(G)$ and the fact that $\trace(R(G))=0$ implies that $\lambda_n=-1$ with multiplicity one in the first case. It can be seen that this requirement is uniquely satisfies by $K_2$. In the second case, it is observed that the corresponding graph is the star, with vertex $v_i$ being the center of the star.
\end{proof}

Using the Cauchy Schwarz Inequality, we have the following theorem.

\begin{theorem}\label{Randic vertex energy upper bound 1}
Let $G=(V,E)$ be a connected graph with at least two vertices, then $\forall$ $v_i\in V$,
$$\mathcal{RE}_G(v_i)\leq\sqrt{\frac{1}{d_i}\sum\limits_{i\sim j}\frac{1}{d_j}},$$
equality holds if and only if $G$ is a star graph, with vertex $v_i$ being the center of the star.
\end{theorem}
\begin{proof}
Let $\lambda_1 \geq \cdots \geq \lambda_n$ are eigenvalues of $R(G)$. From the Cauchy Schwarz Inequality \cite{Bhatia1997}, we get
$$\left(\sum\limits_{j=1}^{n}y^2_{ij}|\lambda_{j}|\right)^2\leq \left(\sum\limits_{j=1}^{n}y^2_{ij}|\lambda_{j}|^2\right)\left(\sum\limits_{j=1}^{n}y^2_{ij}\right).$$
From Lemma \ref{Lemma of iith entry of Randic higher power}, $\sum\limits_{j=1}^{n}y^2_{ij}|\lambda_{j}|^2=\left(R^2(G)\right)_{ii}=\frac{1}{d_i}\sum\limits_{i\sim j}\frac{1}{d_j}$, also $\sum\limits_{j=1}^{n}y^2_{ij}=1$. Hence, from Lemma \ref{Lemma definittion of vetrex energy},
$$\left(\mathcal{RE}_G(v_i)\right)^2\leq \frac{1}{d_i}\sum\limits_{i\sim j}\frac{1}{d_j}.$$

For equality, $y_{ij}|\lambda_j|$ and $y_{ij}$ must be proportional, implying $|\lambda_j|$ is constant for all $j$. Since $1\in\spec(R(G))$ with multiplicity one, the remaining eigenvalues must be $-1$, also simple, to preserve the trace, which holds only for $K_2$. Another equality case arises when $|\lambda_1|=|\lambda_n|=1$, $|\lambda_j|=0$ along with $y_{ij}=0$ for $j=2,\ldots,n-1$. This occurs only for a star graph, with $v_i$ as the center.
\end{proof}

To obtain equality in the above theorem, we must have same absolute value for all non-zero eigenvalues. The next theorem, weakens this condition and identifies graphs for which equality holds.
\begin{theorem}
Let $G=(V,E)$ be a connected graph with $n\geq 2$ vertices and $m$ edges, then $\forall$ $v_i\in V$
$$\mathcal{RE}_G(v_i)\leq \frac{d_i}{2m}+\sqrt{\left(\frac{1}{d_i}\sum\limits_{i\sim j}\frac{1}{d_j}-\frac{d_i}{2m}\right)\left(1-\frac{d_i}{2m}\right)},$$
equality holds if $G\cong K_n$ and $G\cong S_n$, with $v_i$ is its center.
\end{theorem}
\begin{proof}
Let $\lambda_1 \geq \cdots \geq \lambda_n$ are eigenvalues of $R(G)$. From the Cauchy Schwarz Inequality, we have
\begin{align*}
&\left(\sum\limits_{j=1}^{n-1}y^2_{i(j+1)}|\lambda_{j+1}|\right)^2\leq \left(\sum\limits_{j=1}^{n-1}y^2_{i(j+1)}|\lambda_{j+1}|^2\right)\left(\sum\limits_{j=1}^{n-1}y^2_{i(j+1)}\right)\\
\Rightarrow &\left(\mathcal{RE}_G(v_i)-y^2_{i1}|\lambda_{1}|\right)^2\leq \left(\frac{1}{d_i}\sum\limits_{i\sim j}\frac{1}{d_j}-y^2_{i1}|\lambda_{1}|^2\right)\left(1-y^2_{i1}\right).
\end{align*}
Note that $\left(\frac{\sqrt{d_1}}{\sqrt{d_1+\cdots+d_n}},\ldots,\frac{\sqrt{d_n}}{\sqrt{d_1+\cdots+d_n}}\right)^T$ is an orthonormal eigenvector of $R(G)$ corresponding to the eigenvalue $\lambda_1=1$ and hence,
$$\mathcal{RE}_G(v_i)\leq \frac{d_i}{2m}+\sqrt{\left(\frac{1}{d_i}\sum\limits_{i\sim j}\frac{1}{d_j}-\frac{d_i}{2m}\right)\left(1-\frac{d_i}{2m}\right)}.$$

For equality, the magnitudes $|\lambda_{j}|$ have to be constant $\forall$ $j$ except $|\lambda_1|$. Clearly, the complete graph $K_n$ satisfies this criteria. Additionally, the star graph with vertex $v_i$ at the center is also a candidate for achieving the equality.
\end{proof}

The following corollary is a direct consequence of the above theorem.

\begin{corollary}
Let $G=(V,E)$ be a connected $k$-regular graph with $n$ vertices, then $\forall$ $v_i\in V$
$$\mathcal{RE}_G(v_i)\leq \frac{1}{n}+\sqrt{\left(\frac{1}{k}-\frac{1}{n}\right)\left(1-\frac{1}{n}\right)},$$
equality holds if and only if $G\cong K_n$.
\end{corollary}
\begin{proof}
The proof is based on the fact that in a $k$-regular graph, there are $\frac{nk}{2}$ edges.
\end{proof}

\begin{theorem}\label{Randic vertex energy lower bound 1}
Let $G=(V,E)$ be a connected graph with at least two vertices, then $\forall$ $v_i\in V$,
$$\mathcal{RE}_G(v_i)\geq\frac{1}{d_i}\sum\limits_{i\sim j}\frac{1}{d_j},$$
equality holds if and only if $G$ is a complete bipartite graph $K_{n_1,n_2}$.
\end{theorem}
\begin{proof}
Let $\lambda_1\geq\ldots\geq\lambda_n$ be the eigenvalues of $R(G)$, then $-1\leq\lambda_i\leq 1$ $\forall$ $i\in \{1,\ldots,n\}$, which implies $|\lambda_i|\geq \lambda_i^2$. Therefore
$$\mathcal{RE}_G(v_i)=\sum\limits_{j=1}^ny^{2}_{ij}|\lambda_j|\geq \sum\limits_{j=1}^ny^{2}_{ij}\lambda_j^2=\phi_i(R^2)=\frac{1}{d_i}\sum\limits_{i\sim j}\frac{1}{d_j}. $$
The equality holds only when $\lambda_i\in \{-1,0,1\}$ $\forall$ $i\in \{1,\ldots,n\}$. Since 1 is always an eigenvalue of the Randi\'c matrix $R(G)$ with multiplicity one, we seek graphs whose remaining eigenvalues belong to the set $\{-1,0\}$. Suppose, $-1$ is not an eigenvalue of $R(G)$. Then, it contradicts $\trace(R(G))=0$. Therefore, $-1$ also be an eigenvalue, and it must occur with multiplicity one, which implies the graph $G$ must be bipartite. Hence, we are left with a bipartite graph $G$ with $\rank(R(G))=2$. The only graph satisfying these conditions is complete bipartite graph $K_{n_1,n_2}$.
\end{proof}

\begin{corollary}
Among the connected graphs, the pendant vertices of a star have the minimum Randi\'c energy of a vertex.
\end{corollary}
\begin{proof}
The result can be obtained from Theorem \ref{Randic vertex energy lower bound 1} by minimizing the term $\frac{1}{d_i}\sum\limits_{i\sim j}\frac{1}{d_j}$ i.e., by replacing the degree of the neighbourhood vertices of $v_i$ by $n-1$ and hence equality holds for the pendent vertices of a star graph.
\end{proof}

Now, from the H\"{o}lder inequality \cite{Bhatia1997}, we have the following lower bound.

\begin{theorem}
Let $G=(V,E)$ be a connected graph with at least two vertices, then $\forall$ $v_i\in V$
$$\mathcal{RE}_G(v_i)\geq\frac{\left(\frac{1}{d_i}\sum\limits_{i\sim j}\frac{1}{d_j}\right)^{\frac{3}{2}}}{\left(\sum\limits_{j=1}^n\frac{1}{d_id_j}\left(\sum\limits_{k\sim i \& j}\frac{1}{d_k}\right)^2\right)^{\frac{1}{2}}}.$$
\end{theorem}
\begin{proof}
We have $|R(G)|^2=|R(G)|^{\frac{4}{3}}\cdot|R(G)|^{\frac{2}{3}}$, and hence from the H\"{o}lder inequality,
$$\phi_i(|R(G)|^2)\leq \left(\phi_i\left(|R(G)|^{\frac{4}{3}}\right)^p \right)^{\frac{1}{p}} \left(\phi_i\left(|R(G)|^{\frac{2}{3}}\right)^q \right)^{\frac{1}{q}},$$
where $\frac{1}{p}+\frac{1}{q}=1$. In particular taking $p=3$, $q=\frac{3}{2}$ and rearranging, we have
$$\left(\phi_i(|R(G)|^2)\right)^\frac{3}{2}\leq \left(\phi_i\left(|R(G)|^4\right) \right)^{\frac{1}{2}} \mathcal{RE}_G(v_i).$$
Since $\phi_i(|R(G)|^2)=\phi_i(R^2(G))=\frac{1}{d_i}\sum\limits_{i\sim j}\frac{1}{d_j}$ and $\phi_i(|R(G)|^4)=\phi_i(R^4(G))=\sum\limits_{j=1}^n\frac{1}{d_id_j}\left(\sum\limits_{k\sim i \& j}\frac{1}{d_k}\right)^2$, the inequality follows.
\end{proof}

In \cite{Arizmendi2021}, the authors prove that the product of vertex energy of two adjacent vertices is greater than or equal to 1. Using Lemma \ref{Lemma definittion of vetrex energy}, we have the following result for the Randi\'c energy using a similar argument.

\begin{theorem}
Let $G=(V,E)$ be a connected graph. If $v_i$ and $v_j$ are two adjacent vertices of $G$, then
$$\mathcal{RE}_G(v_i)\mathcal{RE}_G(v_j)\geq\frac{1}{d_id_j}.$$
\end{theorem}
\begin{proof}
We have $R(G)=Y\Lambda Y^*$, where $Y \coloneqq(y_{ij})$ is the orthogonal matrix whose columns are given by the eigenvectors of $R(G)$ and $\Lambda\coloneqq \text{diag}\left(\lambda_1,\ldots,\lambda_n\right)$. Consider two vectors $x_1=(y_{i1}\sqrt{|\lambda_1|},\ldots,y_{in}\sqrt{|\lambda_n|})$ and $x_2=(sgn(\lambda_1)y_{j1}\sqrt{|\lambda_1|},\ldots,sgn(\lambda_n)y_{jn}\sqrt{|\lambda_n|})$. If $v_i$ and $v_j$ are adjacent vertices, then
$$\langle x_1,x_2\rangle^2=\left(\sum\limits_{k=1}^ny_{ik}\lambda_ky_{jk}\right)^2=\left((R(G))_{ij}\right)^2=\frac{1}{d_id_j}.$$
Also, $||x_1||^2=\sum\limits_{k}y^{2}_{ik}|\lambda_k|=\mathcal{RE}_G(v_i)$ and $||x_2||^2=\sum\limits_{k}y^{2}_{jk}|\lambda_k|=\mathcal{RE}_G(v_j)$. By Cauchy-schwarz inequality,
\begin{align*}
&\mathcal{RE}_G(v_i)\mathcal{RE}_G(v_j)=||x_1||^2||x_2||^2\geq\langle x_1,x_2\rangle^2=\frac{1}{d_id_j}.\qedhere
\end{align*}
\end{proof}

It is observed that the Randi\'c energy of the entire graph can be effectively estimated from the energies of individual vertices. From Theorems \ref{Randic vertex energy upper bound 1} and \ref{Randic vertex energy lower bound 1}, we can obtain the Randi\'c energy of the graph $G$. Let us first recall the definition of general Randi\'c index, $R^{(\alpha)}(G)$ defined by Bollob\'as et al. \cite{Bollobas1999} in 1999, and it is defined as $$R^{(\alpha)}(G)=\sum\limits_{i\sim j}(d_i d_j)^\alpha.$$

\begin{corollary}
Let $G$ be a connected graph with $n$ vertices, then
$$2R^{(-1)}(G)\leq\mathcal{RE}(G)\leq \sum\limits_{i=1}^n\sqrt{\frac{1}{d_i}\sum\limits_{i\sim j}\frac{1}{d_j}},$$
the lower bound attains if and only if $G\cong K_{n_1,n_2}$ and the upper bound attains if and only if $v_i$ is the central vertex of the star graph.
\end{corollary}
\begin{proof}
The proof is straightforward, since $R^{(-1)}(G)=\frac{1}{2}\sum\limits_{i=1}^n\left(\frac{1}{d_i}\sum\limits_{i\sim j}\frac{1}{d_j}\right)$.
\end{proof}

It is easy to show that the Randi\'c energy of a bipartite graph is distributed equally between the two partitions, which can be proved with the help of the block structure of the Randi\'c matrix, as shown in the following lemma.

\begin{lemma}\label{lemma Randic vertex energy bipartite}
Let $G=(V,E)$ be a bipartite graph with two partitions $V_1$ and $V_2$. Then the total Randi\'c energy of $G$ is equally distributed between the two vertex partitions. i.e.,
$$\sum\limits_{v\in V_1}\mathcal{RE}_G(v)=\sum\limits_{v\in V_2}\mathcal{RE}_G(v)=\frac{1}{2}\mathcal{RE}(G).$$
\end{lemma}
\begin{proof}
Let us consider
$$|R(G)|=\sqrt{\left(
  \begin{array}{c|c}
    0 & X \\
    \hline
    X^t & 0 \\
  \end{array}
\right)\left(
  \begin{array}{c|c}
    0 & X \\
    \hline
    X^t & 0 \\
  \end{array}
\right)^*}=\left(
  \begin{array}{c|c}
    \left(XX^t\right)^{\frac{1}{2}} & 0 \\
    \hline
    0& \left(X^tX\right)^{\frac{1}{2}} \\
  \end{array}
\right).$$
Now, $trace\left(\left(XX^t\right)^{\frac{1}{2}}\right)=\sum\limits_{v_i\in V_1}|R(G)|_{ii}=\sum\limits_{v_i\in V_1}\mathcal{RE}_G(v_i)$ and $trace\left(\left(X^tX\right)^{\frac{1}{2}}\right)=\sum\limits_{v_i\in V_2}|R(G)|_{ii}=\sum\limits_{v_i\in V_2}\mathcal{RE}_G(v_i)$. Note that $\text{trace}\left(XX^t\right)^{\frac{1}{2}}=\text{trace}\left(X^tX\right)^{\frac{1}{2}}$, and hence the result.
\end{proof}

\begin{corollary}
Let $G$ be a bipartite graph with two partitions $V_1$ and $V_2$ with $n_1$ and $n_2$ vertices respectively, then
$$\sum\limits_{v_i\in V_1}\left(\frac{1}{d_i}\sum\limits_{i\sim j}\frac{1}{d_j}\right)\leq \frac{\mathcal{RE}(G)}{2}\leq \sum\limits_{v_i\in V_1}\sqrt{\frac{1}{d_i}\sum\limits_{i\sim j}\frac{1}{d_j}},$$
the lower bound attains if and only if $G\cong K_{n_1,n_2}$ and the upper bound attains if and only if $v_i$ is the central vertex of the star graph.
\end{corollary}
\begin{proof}
The proof follows from Theorems \ref{Randic vertex energy lower bound 1}, \ref{Randic vertex energy upper bound 1}, and Lemma \ref{lemma Randic vertex energy bipartite}.
\end{proof}

Now, we present some other bounds for Randi\'c energy of a vertex by a different combinatorial way. Let us recall that $\binom{\alpha}{k}\coloneqq\frac{\alpha(\alpha -1)\cdots(\alpha-k+1)}{k!}$.

\begin{lemma}
Let $G=(V,E)$ be a connected graph, then $\forall$ $v_i\in V$
$$\mathcal{RE}_G(v_i)=\sum\limits_{k=0}^\infty \binom{\frac{1}{2}}{k}\left(R^2(G)-I\right)^k_{ii}.$$
\end{lemma}
\begin{proof}
The spectral radius of $R(G)$ is 1. It follows that the matrix $\left(R^2(G)-I\right)$ is negative semi definite and has all eigenvalues in the interval $[-1,0]$. Therefore, expanding in power series, we have
$$|R(G)|=\left(I+R^2(G)-I\right)^{\frac{1}{2}}=\sum\limits_{k=0}^\infty \binom{\frac{1}{2}}{k}\left(R^2(G)-I\right)^k.$$
Now, $\mathcal{RE}_G(v_i)=|R(G)|_{ii}$, and hence the result follows.
\end{proof}

Now if we consider the first few terms of the expansion for $|R(G)|$, We have
$$|R(G)|=I+\frac{1}{2}\left(R^2(G)-I\right)-\frac{1}{2\cdot4}\left(R^2(G)-I\right)^2+\frac{1\cdot3}{2\cdot4\cdot6}\left(R^2(G)-I\right)^3-\ldots$$
From this expansion, we have the following proposition.
\begin{proposition}\label{Randic vertex energy upper bound 2}
Let $G$ be a connected graph. Then $\forall$ $v_i\in V$
$$\mathcal{RE}_G(v_i)\leq\frac{1}{2}\left(\frac{1}{d_i}\sum\limits_{i\sim j}\frac{1}{d_j}+1\right).$$
\end{proposition}
\begin{proof}
Since, the matrix $\left(R^2(G)-I\right)$ is negative semi definite, if we consider first two terms in the expansion of $|R(G)|$ then
$$|R(G)|\leq I+\frac{1}{2}\left(R^2(G)-I\right).$$
Again, $\left(R^2(G)-I\right)_{ii}=\frac{1}{d_i}\sum\limits_{i\sim j}\frac{1}{d_j}-1$ implies
\begin{align*}
\mathcal{RE}_G(v_i)=|R(G)|_{ii}&\leq 1+\frac{1}{2}\left(R^2(G)-I\right)_{ii}\\
&=1+\frac{1}{2}\left(\frac{1}{d_i}\sum\limits_{i\sim j}\frac{1}{d_j}-1\right)\\
&=\frac{1}{2}\left(\frac{1}{d_i}\sum\limits_{i\sim j}\frac{1}{d_j}+1\right).\qedhere
\end{align*}
\end{proof}
It is important to highlight that the upper bound presented in Theorem \ref{Randic vertex energy upper bound 1} provides a significantly stronger estimate compared to the bound given in Proposition \ref{Randic vertex energy upper bound 2}. However, the bound in Proposition \ref{Randic vertex energy upper bound 2} is derived using only the first two terms in the expansion of $|R(G)|$. Our aim is simply to illustrate how the energy of a vertex can be estimated using this approach. Naturally, by including more terms from the expansion, we can obtain a much sharper and more accurate upper bound. In the next proposition, we present an upper bound by considering one more term in the expansion of $|R(G)|$.

\begin{proposition}\label{Randic vertex energy upper bound 3}
Let $G$ be a connected graph. Then $\forall$ $v_i\in V$
$$\mathcal{RE}_G(v_i)\leq\frac{3}{8}+\frac{3}{4}\frac{1}{d_i}\sum\limits_{i\sim j}\frac{1}{d_j}-\frac{1}{8}\sum\limits_{j=1}^n\frac{1}{d_id_j}\left(\sum\limits_{k\sim i \& j}\frac{1}{d_k}\right)^2.$$
\end{proposition}
\begin{proof}
It is easy to see that
\begin{align*}
\left(R^2(G)-I\right)^2_{ii}&=\left(R^4(G)+I-2R^2(G)\right)_{ii}\\
&=\sum\limits_{j=1}^n\frac{1}{d_id_j}\left(\sum\limits_{k\sim i \& j}\frac{1}{d_k}\right)^2+1-2\frac{1}{d_i}\sum\limits_{i\sim j}\frac{1}{d_j}.
\end{align*}
Which implies
\begin{align*}
\mathcal{RE}_G(v_i)=|R(G)|_{ii}&\leq 1+\frac{1}{2}\left(R^2(G)-I\right)_{ii}-\frac{1}{8}\left(R^2(G)-I\right)2_{ii}\\
&=1+\frac{1}{2}\left(\frac{1}{d_i}\sum\limits_{i\sim j}\frac{1}{d_j}-1\right)-\frac{1}{8}\left(\sum\limits_{j=1}^n\frac{1}{d_id_j}\left(\sum\limits_{k\sim i \& j}\frac{1}{d_k}\right)^2+1-2\frac{1}{d_i}\sum\limits_{i\sim j}\frac{1}{d_j}\right)\\
&=\frac{3}{8}+\frac{3}{4}\frac{1}{d_i}\sum\limits_{i\sim j}\frac{1}{d_j}-\frac{1}{8}\sum\limits_{j=1}^n\frac{1}{d_id_j}\left(\sum\limits_{k\sim i \& j}\frac{1}{d_k}\right)^2.\qedhere
\end{align*}
\end{proof}

\section{Transitive graphs and their Randi\'c energy}

For a graph $G$, if there is an automorphism $f:G\rightarrow G$ such that $f(v_i)=v_j$ for all pair of vertices $v_i\in V$ i.e., all the vertices are structurally identical, then it is clear that the vertices $v_i$ and $v_j$ have equal Randi\'c energies. For such graphs, the Randi\'c energy of each vertex is simply the total Randi\'c energy divided by the number of vertices, i.e., if $G$ be a graph with $n$ vertices, then for each $i\in \{1,\ldots,n\}$
$$\mathcal{RE}_G(v_i)=\frac{\mathcal{RE}(G)}{n}.$$

Also, there are other graphs, like complete bipartite graph, friendship graph, where similar simplifications are possible, even though the graph may not be vertex-transitive. In this section, we study the Randi\'c energy of such vertex transitive graphs and some other symmetric graphs.

\begin{proposition}
In a complete graph $K_n$, the Randi\'c energy of a vertex is inversely proportional to the number of vertices.
\end{proposition}
\begin{proof}
We have $\spec(R(K_n))=\left(\begin{array}{cc}
                                        1 & \frac{-1}{n-1} \\
                                        1 & n-1
                                      \end{array}\right)$,
implies $\mathcal{RE}(K_n)=1+(n-1)|\frac{-1}{n-1}|=2$. Due to transitivity of the vertices, $\mathcal{RE}_{K_n}(v_i)=\frac{2}{n}.$
\end{proof}
The complete graph $K_5$ and $K_8$ along with the Randi\'c energy of its vertices are given in figure \ref{Fig vertex energy complete graph energy}.\\

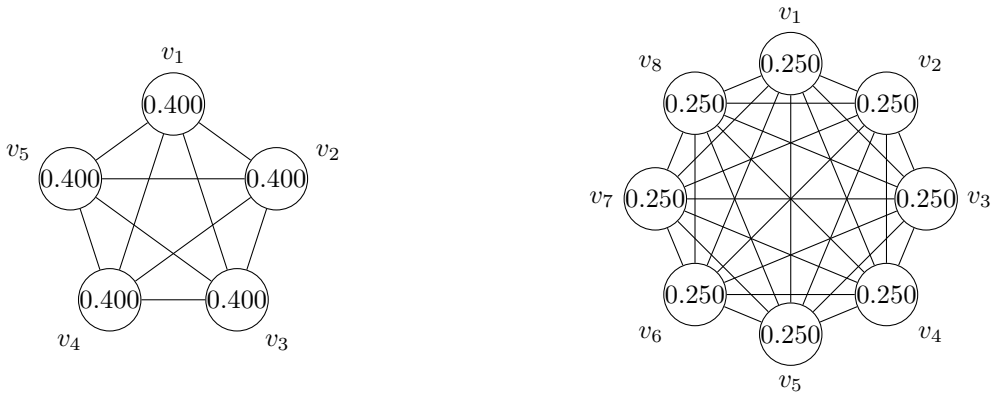
\begin{figure}[H]
\centering
\begin{minipage}{0.45\textwidth}
\centering
\begin{tikzpicture}[scale=.6,vertex/.style={circle,draw,minimum size=8mm,inner sep=0pt}]
\node[vertex,label=90:$v_1$]  (v1) at (90:2.4cm) {0.400};
\node[vertex,label=18:$v_2$]  (v2) at (18:2.4cm) {0.400};
\node[vertex,label=-54:$v_3$] (v3) at (-54:2.4cm) {0.400};
\node[vertex,label=-126:$v_4$](v4) at (-126:2.4cm) {0.400};
\node[vertex,label=162:$v_5$] (v5) at (162:2.4cm) {0.400};

\foreach \a/\b in {1/2,1/3,1/4,1/5, 2/3,2/4,2/5, 3/4,3/5, 4/5}
 \draw (v\a) -- (v\b);
 \end{tikzpicture}
\end{minipage}
\hfill
\begin{minipage}{0.45\textwidth}
\centering
 \begin{tikzpicture}[scale=.6,vertex/.style={circle,draw,minimum size=8mm,inner sep=0pt}]
\node[vertex,label=90:$v_1$]   (v1) at (90:3cm)   {0.250};
\node[vertex,label=45:$v_2$]   (v2) at (45:3cm)   {0.250};
\node[vertex,label=0:$v_3$]    (v3) at (0:3cm)    {0.250};
\node[vertex,label=-45:$v_4$]  (v4) at (-45:3cm)  {0.250};
\node[vertex,label=-90:$v_5$]  (v5) at (-90:3cm)  {0.250};
\node[vertex,label=-135:$v_6$] (v6) at (-135:3cm) {0.250};
\node[vertex,label=180:$v_7$]  (v7) at (180:3cm)  {0.250};
\node[vertex,label=135:$v_8$]  (v8) at (135:3cm)  {0.250};

\foreach \a/\b in {
1/2,1/3,1/4,1/5,1/6,1/7,1/8,
2/3,2/4,2/5,2/6,2/7,2/8,
3/4,3/5,3/6,3/7,3/8,
4/5,4/6,4/7,4/8,
5/6,5/7,5/8,
6/7,6/8,
7/8}
  \draw (v\a) -- (v\b);
 \end{tikzpicture}
\end{minipage}
\caption{Complete Graphs: $K_5$ and $K_8$.}\label{Fig vertex energy complete graph energy}
\end{figure}

\begin{proposition}
Let $C_n$ be the cycle graph with $n$ vertices, then
$$
\mathcal{RE}_{C_n}(v_i)=\left\{
  \begin{array}{ll}
    \frac{2\cos{\frac{\pi}{n}}}{n\sin{\frac{\pi}{n}}}, & \hbox{when $n\equiv 0$ mod(4);} \\
    \frac{1}{n\sin{\frac{\pi}{2n}}}, & \hbox{when $n\equiv 1$ mod(4);}\\
    \frac{2}{n\sin{\frac{\pi}{n}}}, & \hbox{when $n\equiv 2$ mod(4),}
  \end{array}
\right.
$$
for $k\in \{0,\ldots,n-1\}$.\\
\end{proposition}
\begin{proof}
Clearly, $R(C_n)=\frac{1}{2}A(C_n)$, where $A(C_n)$ is the adjacency matrix of $C_n$. Hence, $\cos\left({\frac{2\pi k}{n}}\right)\in\spec(R(C_n))$, $k\in \{0,\ldots,n-1\}$. The total energy can be computed by taking the positive eigenvalues of $R(C_n)$ which leads to the three conditions $n\equiv 0 \; \text{mod}(4)$, $n\equiv 1 \; \text{mod}(4)$ and $n\equiv 2\; \text{mod}(4)$.
\end{proof}

\begin{proposition}
Let $K_{n_1,n_2}$ be the Complete bipartite graph with $V_1$ and $V_2$ be its vertex partitions such that $|V_1|=n_1$ and $|V_2|=n_2$, then
$$\mathcal{RE}_{K_{n_1,n_2}}(v_i)=\left\{
    \begin{array}{ll}
      \frac{1}{n_1}, & \hbox{if $v_i \in V_1$;} \\
      \frac{1}{n_2}, & \hbox{if $v_i \in V_2$.}
    \end{array}
  \right.
$$
\end{proposition}
\begin{proof}
Clearly, $\mathcal{RE}(K_{n_1,n_2})=2$. Since the energy of a bipartite graph is distributed equally in its two partitions, this implies that each partition has total energy 1, and hence the result.
\end{proof}

The star graph $S_n=K_{1,n-1}$ with $n$ vertices can be considered as a particular case of a complete bipartite graph. Therefore,
$$\mathcal{RE}_{S_n}(v_i)=\left\{
    \begin{array}{ll}
      \;\;1, & \hbox{if $v_i$ is the centre of the star;} \\
      \frac{1}{n-1}, & \hbox{otherwise.}
    \end{array}
  \right.
$$

The friendship graph $F_n$ can be constructed by joining $n$ copies of the cycle graph $C_3$ with a common vertex i.e., every two vertices have exactly one neighbor in common. Even though it is not edge-transitive, the graph has enough regularity and hence we make the following proposition.

\begin{proposition}
Let $F_n$ be the Friendship graph with $n$ vertices, then
$$\mathcal{RE}_{F_n}(v_i)=\left\{
    \begin{array}{ll}
      \;\;\frac{2}{3}, & \hbox{if $v_i$ is the common vertex;} \\
      \frac{3n+1}{6n}, & \hbox{otherwise.}
    \end{array}
  \right.
$$
\end{proposition}
\begin{proof}
The Randi\'c matrix of the friendship graph is of the form
$$
R(F_n)=\left(
  \begin{array}{c|c}
    0 & \frac{1}{2\sqrt{n}}{\bf{1}}^t_{2n} \\
\hline
    \frac{1}{2\sqrt{n}}{\bf{1}}_{2n} & \frac{1}{2}I_n\otimes \left[
                                                               \begin{array}{cc}
                                                                 0 & 1 \\
                                                                 1 & 0 \\
                                                               \end{array}
                                                             \right]
 \\
  \end{array}
\right)
$$
Let us consider the following characteristic equation
\begin{align*}
&|\lambda I-R(F_n)|=0\\
\Rightarrow&\left|\lambda I-\frac{1}{2}I_n\otimes \left[
                                                               \begin{array}{cc}
                                                                 0 & 1 \\
                                                                 1 & 0 \\
                                                               \end{array}
                                                             \right]\right|\cdot \left|\lambda I-\frac{1}{4n}{\bf{1}}^t\left(\lambda I-\frac{1}{2}I_n\otimes \left[
                                                               \begin{array}{cc}
                                                                 0 & 1 \\
                                                                 1 & 0 \\
                                                               \end{array}
                                                             \right]\right)^{-1}\bf{1}\right|=0.
\end{align*}
Upon solving we get
$\spec(R(F_n))=\left(\begin{array}{ccc}
    1 & \frac{1}{2} & -\frac{1}{2} \\
    1 & n-1 & n+1
  \end{array}\right)$, makes $\mathcal{RE}(F_n)=n+1$.

To compute the energy of the common vertex, say $v_1$ we evaluate the weights of the eigenvalues corresponding to $v_1$. The energy of $v_1$ is then given by the sum over all eigenvalues of the product of each weight with the absolute value of the corresponding eigenvalue. For the remaining $2n$ vertices, their energies are equal due to symmetry, and hence $$\mathcal{RE}_{F_n}(v_i)=\frac{\mathcal{RE}(F_n)-\mathcal{RE}_{F_n}(v_1)}{2n}.$$
Since $R(F_n)$ has only three distinct eigenvalues, say $\lambda_1\geq\lambda_2\geq\lambda_3$, the corresponding weights $y^2_{1i}$ can be determined from the following system of equations:
\begin{align*}
&x_{11}\lambda^0_1+x_{12}\lambda^0_2+x_{13}\lambda^0_3=1\\
&x_{11}\lambda^1_1+x_{12}\lambda^1_2+x_{13}\lambda^1_3=0\\
&x_{11}\lambda^2_1+x_{12}\lambda^2_2+x_{13}\lambda^2_3=\frac{1}{2},
\end{align*}
where $x_{11}=y_{11}^2$, $x_{12}=\sum\limits_{j=2}^{n}y_{1j}^2$ and $x_{13}=\sum\limits_{j=n+1}^{2n+1}y_{1j}^2$. Solving, we get $x_{11}=\frac{1}{3}$, $x_{12}=0$ and $x_{13}=\frac{2}{3}$.
\end{proof}

\section{Coulson Integral Formula for Randi\'{c} Energy of a Vertex}
In many cases, the computation of graph energy directly from the eigenvalues of the graph is neither efficient nor feasible. In such cases, the Coulson integral formula offers an alternative approach to compute the energy of a graph using only its characteristic polynomial, without requiring explicit properties of the eigenvalues. Arizmendi et al. \cite{Arizmendi2019} in 2019, provided a refinement of Coulson integral formula to energy of a vertex in relation with the characteristic polynomial of the graph. Notably, this formula is not restricted to the adjacency matrix but extends to all the hermitian matrices of a graph. In this section, we report the Coulson integral formula for the Randi\'c energy of a vertex.

To apply the Coulson integral formula to compute the Randi\'c energy of a vertex, it is essential to know the characteristic polynomial $\psi_R(G;x)$ of the Randi\'c matrix. There are extensive literature about the $\psi_R(G;x)$, not restricted to the undirected simple graphs but also directed and mixed graphs.

Let $\psi_R(G;x)\coloneqq det(xI-R(G))$ be the characteristic polynomial of $R(G)$ of a simple connected graph $G$. Then the expression for the coefficients of $\psi_R(G;x)$ is given in the following theorem.

\begin{theorem}\cite{Lu2017, Bharali}
If $\psi_R(G;x)\coloneqq x^n+a_1x^{n-1}+\cdots+a_n$, then
$$(-1)^ka_k=\sum\limits_{G'}(-1)^{n-c(G')}2^{s(G'))}\prod_{v_i\in V(G')}\frac{1}{d_i},$$
where the summation is over all elementary sub-graphs $G'$ with order $k$ of $G$, $c(G')$ be the number of components of $G'$, $s(G')$ be the number of cycles of length at least 3 in $G'$.
\end{theorem}
The proof of the above theorem is based on the fact that the summation of the determinants of all principal $k\times k$ sub-matrices of $R(G)$ is $(-1)^ka_k$. If $G$ be a bipartite graph then its spectrum is symmetric about 0 and hence, its characteristic polynomial is of the form
$$\sum_{k=0}^{\lfloor n/2 \rfloor} (-1)^kb_{2k}x^{n-2k},$$
where $b_{2k} \geq 0$, $\forall$ $k$. \\

\begin{theorem}\label{Coulson integral formula to energy of a vertex}
Let $G=(V,E)$ be a graph. If $\psi_R(G;z)$ and $\psi_R(G-v_i;z)$ be the characteristic polynomial of $R(G)$ and $R(G-v_i)$, respectively then
$$\mathcal{RE}_{G}(v_i) = \frac{1}{\pi} \int_{\mathbb{R}} \left(1 - \frac{\mathbf{i}x \psi_R(G-v_i;\mathbf{i}x) }{\psi_R(G;\mathbf{i}x)}\right)dx,\;\;\;\;\forall\; v_i\in V.$$
\end{theorem}

The proof is obtained by employing a similar approach as in \cite{Arizmendi2019}.

Now, we can use the quasi-order technique to compare the Randi\'c energies of different vertices. For two bipartite graphs $G_1$ and $G_2$, define the quasi-order as
$$G_1 \preceq G_2, \;\;\;\;\;\;\;\text{if} \;\; b_{2k}(G_1) \leq b_{2k}(G_2)\;\;\;\forall\;k.$$

In the following theorem we present a comparison of Randi\'c energies between two vertices in a graph.
\begin{theorem}
Let $G$ be a bipartite graph and $v, w \in V(G)$ such that $G - w$ and $G - v$ remains bipartite. If $G - w \succeq G - v$, then
$$\mathcal{RE}_{G}(w)\leq \mathcal{RE}_{G}(v),$$
equality holds if and only if $G - w$ and  $G - v$ are co-spectral.
\end{theorem}
\begin{proof}
Being bipartite graphs, let the characteristic polynomials of $G$, $G - w$ and $G - v$ are respectively,
$\phi_R(G; x) = \sum\limits_{k =0}^{\lfloor n/2 \rfloor} (-1)^k a_{2k} x^{n-2k}$,
$\phi_R(G - w; x) = \sum\limits_{k=0}^{\lfloor (n-1)/2 \rfloor} (-1)^k b_{2k} x^{n-1-2k}$,
and $\phi_R(G - v; x) = \sum\limits_{k=0}^{\lfloor (n-1)/2 \rfloor} (-1)^k c_{2k} x^{n-1-2k}$. From Theorem \ref{Coulson integral formula to energy of a vertex} we have
$$\mathcal{RE}_{G}(w)= \frac{1}{\pi} \int_{R} \left(1 - \frac{\mathbf{i}x \, \phi_R(G - w; \mathbf{i}x)}{\phi_R(G; \mathbf{i}x)}\right)  dx = \frac{1}{\pi} \int_{R} \left(1 - \frac{\sum\limits_{k = 0}^{\lfloor (n-1)/2 \rfloor} b_{2k} x^{-2k}}{\sum\limits_{k = 0}^{\lfloor n/2 \rfloor} a_{2k} x^{-2k}}\right) dx. $$
Again, $G - w \succeq G - v$ follows that $b_{2k} \geq c_{2k}$, $\forall$ $k$, which implies

$$\mathcal{RE}_{G}(w)\leq \frac{1}{\pi} \int_{R} \left(1 - \frac{\sum\limits_{k = 0}^{\lfloor (n-1)/2 \rfloor} c_{2k} x^{-2k}}{\sum\limits_{k = 0}^{\lfloor n/2 \rfloor} a_{2k} x^{-2k}}\right)  dx = \mathcal{RE}_{G}(v).$$
For equality we must have $b_{2k}= c_{2k}$ for all $k$.
\end{proof}

Now, we obtain the following theorem for disjoint union of two graphs.

\begin{theorem}
Let $G_1$ and $G_2$ be two disjoint bipartite graphs with $w \in V(G_2)$ and $v \in V(G_1)$. If $G_1 \cup (G_2 - w) \;\succeq\; G_2 \cup (G_1 - v)$, then
$$\mathcal{RE}_{G_2}(w)\leq \mathcal{RE}_{G_1}(v).$$
\end{theorem}
\begin{proof}
Since $G_1 \cup (G_2 - w) = (G_1 \cup G_2) - w $ and $G_2 \cup (G_1 - v) = (G_2 \cup G_1) - v$, taking $G = G_1 \cup G_2$, we obtain
$$ \mathcal{RE}_{G_1\cup G_2}(w)\leq \mathcal{RE}_{G_1\cup G_2}(v).$$
Also, $\mathcal{RE}_{G_1 \cup G_2}(w) = \mathcal{RE}_{G_2}(w)$ and $\mathcal{RE}_{G_1 \cup G_2}(v) = \mathcal{RE}_{G_1}(v)$, which completes the proof.
\end{proof}

\end{document}